\documentclass{amsart}
\usepackage{amsfonts,amssymb,amsmath,amsthm}
\usepackage{url}
\usepackage{enumerate}

\newtheorem{thrm}{Theorem}[section]
\newtheorem{lem}[thrm]{Lemma}
\newtheorem{prop}[thrm]{Proposition}
\newtheorem{cor}[thrm]{Corollary}
\theoremstyle{definition}
\newtheorem{definition}[thrm]{Definition}

\numberwithin{equation}{section}

\author{Na Huang and Jingjing Xue}
\address{Department of Applied Mathematics, Northwestern Polytechnical University, Xi'an, Shaanxi,
710129, P. R. China } \email{huangna7@126.com}
\thanks{This work was supported by the National Natural Science Foundation of China (Grant Nos. 11271299, 11001221) and Natural Science Foundation Research Project of Shaanxi Province (2012JM1014)}
\begin{document}

\title[Running Head]{Inequalities of Dirichlet eigenvalues for degenerate elliptic partial differential operators}

\begin{abstract}
Let ${X_j},{Y_j}(j = 1, \cdot  \cdot  \cdot ,n)$ be vector fields
satisfying H\"{o}rmander's condition and ${\Delta _L} =
\sum\limits_{j = 1}^n {(X_j^2 + Y_j^2)}$. In this paper, we
establish some inequalities of Dirichlet eigenvalues for degenerate
elliptic partial differential operator ${\Delta _L}$ and $\Delta
_L^2$. These inequalities extend Yang's inequalities for Dirichlet
eigenvalues of Laplacian to the settings here and the forms of
inequalities are more general than Yang's inequalities. To obtain
them, we give a generalization of the inequality by Chebyshev.
\end{abstract}
\maketitle

\section{Introduction} \label{sect1}
Estimates of Dirichlet eigenvalues for Laplacian in the Euclidean
space have been extensively studied. For the following Dirichlet
problem\begin{displaymath}
 \left\{ {\begin{array}{*{20}c}
   { - \Delta  u  = \lambda u,} \hfill & {\mbox{in}~~ \Omega ,} \hfill  \\
   {u = 0,} \hfill & { \mbox{on}~~ \partial \Omega ,} \hfill  \\
\end{array}} \right.
\end{displaymath}where
$\Omega $ is a bounded domain in ${R^n}$, Payne, P\'{o}lya and
Weinberger in [11] obtained the inequality (now called the PPW
inequality)
\[{\lambda _{k + 1}} - {\lambda _k} \le \frac{4}{{nk}}\sum\limits_{r = 1}^k {{\lambda
_r}}.\]Hile and Protter in [4] proved the inequality (now called the
HP inequality)
\[\sum\limits_{r = 1}^k {\frac{{{\lambda _r}}}{{{\lambda _{k + 1}} - {\lambda _r}}}}  \ge \frac{{nk}}{4}.\]
Recently, Yang in [13] established some important eigenvalue
estimates including Yang's first inequality
\[\sum\limits_{r = 1}^k {{{\left( {{\lambda _{k + 1}} - {\lambda _r}} \right)}^2}}  \le \frac{4}{n}\sum\limits_{r = 1}^k {\left( {{\lambda _{k + 1}} - {\lambda _r}} \right)} {\lambda _r}\]
and Yang's second inequality
\[{\lambda _{k + 1}} \le \left( {1 + \frac{4}{n}} \right)\frac{1}{k}\sum\limits_{r = 1}^k {{\lambda _r}}.\]

\label{sect1}Some estimates for Dirichlet eigenvalues of
sub-Laplacian on the Heisenberg group was deduced. Niu and Zhang in
[10] obtained the PPW type inequality:
\[{\lambda _{k + 1}} - {\lambda _k} \le \frac{2}{{nk}}\left( {\sum\limits_{r = 1}^k {{\lambda _r}} } \right).\]Ilias and Makhoul in [5] gave the Yang type inequalities.

\label{sect1}In the paper, we consider the following two Dirichlet
problems:\begin{equation}
 \left\{ {\begin{array}{*{20}c}
   { - \Delta_L  u  = \lambda u,} \hfill & {\mbox{in}~~ \Omega ,} \hfill  \\
   {u = 0,} \hfill & { \mbox{on}~~ \partial \Omega ,} \hfill  \\
\end{array}} \right.
\end{equation}
and
\begin{equation}
 \left\{ {\begin{array}{*{20}c}
   {{(-\Delta_L)}^2 u= \lambda u,}  \hfill & {\mbox{in}~~ \Omega ,} \hfill  \\
   {u =\frac{\partial u}{\partial\nu}= 0,} \hfill & { \mbox{on}~~ \partial \Omega ,} \hfill  \\
\end{array}} \right.
\end{equation}
where $\Omega  \subset {R^{2n + 1}}$ is a bounded domain, the
boundary $\partial \Omega $ is smooth and not characteristic, $\nu $
is the outward unit normal on $\partial \Omega $; ${\Delta _L}$ is
the degenerate elliptic partial differential operator constituted by
vector fields ${X_j},{Y_j}(j = 1, \cdot  \cdot  \cdot ,n)$
satisfying H\"{o}rmander's condition,
\begin{align}{\Delta _L} =
\sum\limits_{j = 1}^n {(X_j^2 + Y_j^2)},
\end{align}
where ${X_j} = \frac{\partial }{{\partial {x_j}}} + 2\sigma
{y_j}{\left| z \right|^{2\sigma  - 2}}\frac{\partial }{{\partial
t}}$, ${Y_j} = \frac{\partial }{{\partial {y_j}}} - 2\sigma
{x_j}{\left| z \right|^{2\sigma  - 2}}\frac{\partial }{{\partial
t}}$, $j = 1, \cdot  \cdot  \cdot ,n,$ $x,y \in {R^n}$, $t \in R$,
$z = x + \sqrt { - 1} y \in C$, $\left| z \right| = {\left[
{\sum\limits_{j = 1}^n {(x_j^2 + y_j^2)} } \right]^{\frac{1}{2}}}$,
$\sigma $ is any natural number. When $\sigma  = 1$, ${\Delta _L}$
is the sub-Laplacian on the Heisenberg group; when $\sigma  = 2,3,
\cdot  \cdot  \cdot $, ${\Delta _L}$ is the operators discussed by
Greiner (see [3, 8]). We note that compared with sub-Laplacian on
the Heisenberg group, those operators by Greiner do not have
properties of group structure and translation. Some related papers
see [9, 14].

 \label{sect1}From [7], we know that
the eigenvalues of (1.1) and (1.2) exist and satisfy
\[0 < {\lambda _1} \le {\lambda _2} \le  \cdot  \cdot  \cdot  \le {\lambda _k} \le  \cdot  \cdot  \cdot  \to  + \infty.\]
The corresponding orthogonal normalized eigenfunctions ${u_1},{u_2},
\cdot  \cdot  \cdot ,{u_k}, \cdot  \cdot  \cdot $ satisfy
$\left\langle {{u_i},{u_l}} \right\rangle  = {\delta _{il}},$ $i,l =
1,2, \cdot  \cdot  \cdot.$ Since the boundary $\partial \Omega $ is
not characteristic, the eigenfunctions are smooth by using the
results in [12].

For convenience, we denote $L =  - {\Delta _L}$ in the sequel. The
main results of this paper are the following:

\begin{thrm} \label{l1}
Let $\left\{ {{\lambda _i}} \right\}$ be the eigenvalues of (1.1),
then
\begin{equation}\sum\limits_{i = 1}^k {({\lambda _{k + 1}} - {\lambda _i}} {)^\alpha } \le \sqrt {\frac{2}{n}} {\left( {\sum\limits_{i = 1}^k {({\lambda _{k + 1}} - {\lambda _i}} {)^\beta }\sum\limits_{i = 1}^k {{{\left( {{\lambda _{k + 1}} - {\lambda _i}} \right)}^{2\alpha  - \beta  - 1}}} {\lambda _i}}
\right)^{\frac{1}{2}}}.
\end{equation}\end{thrm}
\noindent where $\alpha  \in R,\beta  \ge 0$ and ${\alpha ^2} \le
2\beta $.

Inequality (1.4) is the generalization of Yang Type inequalities.
Using Theorem 1, it follows some interesting corollaries.
\begin{cor} \label{l1}
Let $\left\{ {{\lambda _i}} \right\}$ be the eigenvalues of (1.1),
then we have the Yang type first inequality
\begin{equation}\sum\limits_{i = 1}^k {({\lambda _{k + 1}} - {\lambda
_i}} {)^2} \le \frac{2}{n}\sum\limits_{i = 1}^k {({\lambda _{k + 1}}
- {\lambda _i}} ){\lambda _i}.\end{equation}
\end{cor}
\begin{cor} \label{l1}
Let $\left\{ {{\lambda _i}} \right\}$ be the eigenvalues of (1.1),
then we have the Payne-P\'{o}lya-Weinberger Type inequality
\begin{equation}{\lambda _{k + 1}} - {\lambda _k} \le \frac{2}{{nk}}\sum\limits_{i = 1}^k {{\lambda _i}}.
\end{equation}
\end{cor}
\begin{cor} \label{l1}
Let $\left\{ {{\lambda _i}} \right\}$ be the eigenvalues of (1.1),
then we have the Yang type second inequality
\begin{equation}{\lambda _{k + 1}} \le \left( {1 + \frac{2}{n}} \right)\frac{1}{k}\left( {\sum\limits_{i = 1}^k {{\lambda _i}} } \right).
\end{equation}
\end{cor}
\begin{thrm}\label{t1}
Let $\left\{ {{\lambda _i}} \right\}$ be the eigenvalues of (1.2),
then\begin{align}&\sum\limits_{i = 1}^k {({\lambda _{k + 1}} -
{\lambda _i}} {)^\alpha }
\\
 &\le
\frac{{2\sqrt {n + 1} }}{n}{\left[ {\sum\limits_{i = 1}^k
{{{({\lambda _{k + 1}} - {\lambda _i})}^\beta }} \lambda
_i^{\frac{1}{2}}} \right]^{\frac{1}{2}}}{\left[ {\sum\limits_{i =
1}^k {{{({\lambda _{k + 1}} - {\lambda _i})}^{2\alpha  - \beta  -
1}}} \lambda _i^{\frac{1}{2}}} \right]^{\frac{1}{2}}}.\notag
\end{align}
\end{thrm}\noindent where $\alpha  \in R,\beta  \ge 0,$ and ${\alpha ^2} \le 2\beta
$.
\begin{cor} \label{l1}
Let $\left\{ {{\lambda _i}} \right\}$ be the eigenvalues of (1.2),
then
\begin{equation}\sum\limits_{i = 1}^k {({\lambda _{k + 1}} - {\lambda _i}} {)^2} \le \frac{{2\sqrt {n + 1} }}{n}{\left[ {\sum\limits_{i = 1}^k {({\lambda _{k + 1}} - {\lambda _i}} )\lambda _i^{\frac{1}{2}}} \right]^{\frac{1}{2}}}{\left[ {\sum\limits_{i = 1}^k {{{({\lambda _{k + 1}} - {\lambda _i})}^2}} \lambda _i^{\frac{1}{2}}} \right]^{\frac{1}{2}}}.
\end{equation}
\end{cor}
\begin{cor} \label{l1}
Let $\left\{ {{\lambda _i}} \right\}$ be the eigenvalues of (1.2),
then
\begin{align}&\sum\limits_{i = 1}^k {({\lambda _{k + 1}} - {\lambda _i}} {)^\alpha }
\\
&\le \frac{{2\sqrt {n + 1} }}{n}{\left[ {\sum\limits_{i = 1}^k
{{{({\lambda _{k + 1}} - {\lambda _i})}^\beta }} }
\right]^{\frac{1}{2}}}{\left[ {\sum\limits_{i = 1}^k {{{({\lambda
_{k + 1}} - {\lambda _i})}^{2\alpha  - \beta  - 1}}{\lambda _i}} }
\right]^{\frac{1}{2}}}.\notag
\end{align}
\end{cor}\noindent where $\alpha  \in R,\beta  \ge 0,$ and ${\alpha ^2} \le 2\beta
$.
\begin{cor} \label{l1}
Let $\left\{ {{\lambda _i}} \right\}$ be the eigenvalues of (1.2),
then we have
\begin{equation}\sum\limits_{i = 1}^k {({\lambda _{k + 1}} - {\lambda _i}} {)^2} \le \frac{{4\left( {n + 1} \right)}}{{{n^2}}}\sum\limits_{i = 1}^k {({\lambda _{k + 1}} - {\lambda _i})} {\lambda _i}.
\end{equation}
\end{cor}
\begin{cor} \label{l1}
Let $\left\{ {{\lambda _i}} \right\}$ be the eigenvalues of (1.2),
then we have
\begin{equation}{\lambda _{k + 1}} - {\lambda _k} \le \frac{{4\left( {n + 1} \right)}}{{{n^2}{k^2}}}{\left( {\sum\limits_{i = 1}^k {\lambda _i^{\frac{1}{2}}} } \right)^2}.
\end{equation}
\end{cor}

These results are new even for Laplacian on the Euclidean space and
sub-Laplacian on the Heisenberg group.

This paper is arranged as follows. In Section 2 the definition of
function couple ${\chi _\lambda }$ and its properties are given; two
elementary inequalities (see Lemmas 2.6 and 2.8) are proved and
examples of noncharacteristics and characteristics domains for
vector fields are listed. The proofs of Theorem 1.1 and Corollaries
1.2-1.4 are put in Section 3. The proofs of Theorem 1.5 and
Corollaries 1.6-1.9 are given in Section 4.
\section{Preliminary results} \label{ns}
\begin{definition}\label{d1}(see [5]) A couple $(f,g)$ of functions on the
interval $(0,\lambda )$ ($\lambda  > 0$) is said to belong to ${\chi
_\lambda }$ provided that
\begin{enumerate}[(i)]
\item $f$ and $g$ are positive. \item $f$ and $g$ satisfy \[{\left( {\frac{{f(x) -
f(y)}}{{x - y}}} \right)^2} + {\left( {\frac{{{{\left( {f(x)}
\right)}^2}}}{{g(x)(\lambda  - x)}} + \frac{{{{\left( {f(y)}
\right)}^2}}}{{g(y)(\lambda  - y)}}} \right)}\left( {\frac{{g(x) -
g(y)}}{{x - y}}} \right) \le 0,\]
\end{enumerate}
\end{definition}
\noindent for any $x,y \in (0,\lambda )$, $x \ne y$.
\begin{lem} \label{l1}
Let $(f,g) \in {\chi _\lambda }$, then $g$ must be nonincreasing; if
$f(x) = {(\lambda  - x)^\alpha }$, $g(x) = {(\lambda  - x)^\beta }$,
then ${\alpha ^2} \le 2\beta $.
\end{lem}
\begin{proof}
From Definition 2.1 we see that $g$ must be nonincreasing. Because
$f$ and $g$ satisfy\[{\left( {\frac{{f(x) - f(y)}}{{x - y}}}
\right)^2} + {\left( {\frac{{{{\left( {f(x)}
\right)}^2}}}{{g(x)(\lambda  - x)}} + \frac{{{{\left( {f(y)}
\right)}^2}}}{{g(y)(\lambda  - y)}}} \right)}\left( {\frac{{g(x) -
g(y)}}{{x - y}}} \right) \le 0,\] letting $y \to x$, we have
\[{\left( {f'(x)} \right)^2} + \frac{{2{{\left( {f(x)}
\right)}^2}}}{{g(x)(\lambda  - x)}}g'(x) \le 0\]and then \[{\left(
{\frac{{f'(x)}}{{f(x)}}} \right)^2} + \frac{2}{{(\lambda  -
x)}}\frac{{g'(x)}}{{g(x)}} \le 0.\]Taking $f(x) = {(\lambda  -
x)^\alpha }$, $g(x) = {(\lambda  - x)^\beta }$, it follows ${\alpha
^2} \le 2\beta $.
\end{proof}
\begin{definition}(see [5]) For any two operators $A$ and
$B$, their commutator $\left[ {A,B} \right]$ is defined by $\left[
{A,B} \right] = AB - BA$.
\end{definition}
\begin{lem}For $p = 1,2, \cdot  \cdot  \cdot ,n,$ we have \begin{equation}L({x_p}{u_i}) = {x_p}L{u_i} - 2{X_p}{u_i},\end{equation}
\begin{equation}\left[ {L,{x_p}} \right]{u_i} =  - 2{X_p}{u_i}.\end{equation}
\end{lem}
\begin{proof}
A direct calculation gives\[{X_j}({x_p}{u_i}) = ({X_j}{x_p}){u_i} +
{x_p}({X_j}{u_i}),\]
\begin{equation}\begin{aligned} X_j^2({x_p}{u_i})&={X_j}(({X_j}{x_p}){u_i} +
{x_p}({X_j}{u_i}))
\\
&=2({X_j}{x_p})({X_j}{u_i}) +
{x_p}(X_j^2{u_i}).\end{aligned}\notag\end{equation} and
\[{Y_j}({x_p}{u_i}) = {x_p}({Y_j}{u_i}).\]\[Y_j^2({x_p}{u_i}) = {x_p}(Y_j^2{u_i}).\]
So \begin{equation}\begin{aligned}L({x_p}{u_i}) &=- \sum\limits_{j =
1}^n {(X_j^2 + Y_j^2)} ({x_p}{u_i})
\\
 &=- \sum\limits_{j = 1}^n {\left[ {2({X_j}{x_p})({X_j}{u_i}) + {x_p}(X_j^2{u_i}) + {x_p}(Y_j^2{u_i})} \right]}
\\
 &={x_p}L{u_i} - 2{X_p}{u_i},
\end{aligned}\notag\end{equation}
and (2.1) is proved. Noting \[\left[ {L,{x_p}} \right]{u_i} =
L({x_p}{u_i}) - {x_p}L{u_i} = {x_p}L{u_i} - 2{X_p}{u_i} -
{x_p}L{u_i} =  - 2{X_p}{u_i},\](2.2) is proved.
\end{proof}
\begin{lem}
(see [5]) Let $A$: $D \subset H \to H$ be a self-adjoint operator
defined on a dense domain $D$, which is semibounded below and has a
discrete spectrum ${\lambda _1} \le {\lambda _2} \le {\lambda _3}
\cdot  \cdot  \cdot $. Let $\left\{ {{T_p}:D \to H} \right\}_{p =
1}^n$ be a collection of skew-symmetric operators, and $\left\{
{{B_p}:{T_p}(D) \to H} \right\}_{p = 1}^n$ be a collection of
symmetric operators, leaving $D$ invariant. We denote by $\left\{
{{u_i}} \right\}_{i = 1}^n$ a basis of orthonormal eigenvectors of
$A$, ${u_i}$ corresponding to ${\lambda _i}$. Let $k \ge 1$ and
assume ${\lambda _{k + 1}} \ge {\lambda _k}$. Then for any $(f,g)$
in ${\chi _{{\lambda _{k + 1}}}}$, it follows
\begin{align}
&{\left( {\sum\limits_{i = 1}^k {\sum\limits_{p = 1}^n {f({\lambda
_i})\left\langle {\left[ {{T_p},{B_p}} \right]{u_i},{u_i}}
\right\rangle } } } \right)^2}
\\
&\le 4\left( {\sum\limits_{i = 1}^k {\sum\limits_{p = 1}^n
{g({\lambda _i})\left\langle {\left[ {A,{B_p}}
\right]{u_i},{B_p}{u_i}} \right\rangle } } } \right)\left(
{\sum\limits_{i = 1}^k {\sum\limits_{p = 1}^n {\frac{{{{(f({\lambda
_i}))}^2}}}{{g({\lambda _i})({\lambda _{k + 1}} - {\lambda
_i})}}{{\left\| {{T_p}{u_i}} \right\|}^2}} } } \right)\notag.
\end{align}
\end{lem}
\begin{lem}
For $\gamma  \ge 1$, ${s_i} \ge 0,i = 1, \cdot  \cdot  \cdot ,k$, we
have\[{\left( {\sum\limits_{i = 1}^k {{s_i}} } \right)^\gamma } \le
{k^{\gamma  - 1}}\sum\limits_{i = 1}^k {s_i^\gamma }. \]
\end{lem}
\begin{proof}
Let $\theta (s) = {s^\gamma },s \ge 0,\gamma  \ge 1$, so $\theta
'(s) = \gamma {s^{\gamma  - 1}} \ge 0,\theta ''(s) = \gamma (\gamma
- 1){s^{\gamma  - 2}} \ge 0$. Noting that $\theta (s)$ is a convex
function on $(0, + \infty )$, we have that for ${s_i} > 0,i = 1,
\cdot  \cdot \cdot ,k$, it holds
\[\theta \left( {{\raise0.7ex\hbox{${\sum\limits_{i = 1}^k {{s_i}} }$} \!\mathord{\left/
 {\vphantom {{\sum\limits_{i = 1}^k {{s_i}} } k}}\right.\kern-\nulldelimiterspace}
\!\lower0.7ex\hbox{$k$}}} \right) \le \frac{{\sum\limits_{i = 1}^k
{\left( {\theta ({s_i})} \right)} }}{k},\] and yields\[{\left(
{\frac{{\sum\limits_{i = 1}^k {{s_i}} }}{k}} \right)^\gamma } \le
\frac{{\sum\limits_{i = 1}^k {s_i^\gamma } }}{k}.\] The required
inequality is proved.
\end{proof}
\begin{lem}
(Chebyshev's inequality, [6]) If $\left( {{a_k} - {a_j}}
\right)\left( {{b_k} - {b_j}} \right) \le 0$ for any nonnegative
$k$, $g$, then
\[\sum\limits_{i = 1}^n {{a_i}{b_i}}  \le \frac{1}{n}\left( {\sum\limits_{i = 1}^n {{a_i}} } \right)\left( {\sum\limits_{i = 1}^n {{b_i}} } \right).\]
\end{lem}

A key preliminary inequality in the paper is the following which
enable us to obtain estimates of eigenvalues more general than
Yang's.
\begin{lem}
If ${A_1} \ge {A_2} \ge  \cdot  \cdot  \cdot  \ge {A_k} \ge 0,$ $0
\le {B_1} \le {B_2} \le  \cdot  \cdot  \cdot  \le {B_k},$ $0 \le
{C_1} \le {C_2}$$\le  \cdot  \cdot  \cdot $$ \le {C_k},$ $i = 1,
\cdot  \cdot  \cdot ,k$, then for ${\alpha ^2} \le 2\beta $, we
have\begin{equation}\sum\limits_{i = 1}^k {A_i^\beta {B_i}}
\sum\limits_{i = 1}^k {A_i^{2\alpha  - \beta  - 1}{C_i}}  \le
\sum\limits_{i = 1}^k {A_i^\beta } \sum\limits_{i = 1}^k
{A_i^{2\alpha  - \beta  - 1}{B_i}{C_i}}. \end{equation}
\end{lem}
\begin{proof}
When $k = 1$, we see that (2.4) is true, since $A_1^\beta
{B_1}A_1^{2\alpha  - \beta  - 1}{C_1} - A_1^\beta A_1^{2\alpha  -
\beta  - 1}{B_1}{C_1} = 0$. Now suppose that the conclusion is true
for $k - 1$, then\[\sum\limits_{i = 1}^k {A_i^\beta {B_i}}
\sum\limits_{i = 1}^k {A_i^{2\alpha  - \beta  - 1}{C_i}}  -
\sum\limits_{i = 1}^k {A_i^\beta } \sum\limits_{i = 1}^k
{A_i^{2\alpha  - \beta  - 1}{B_i}{C_i}} \]\[\begin{array}{l}
 = \sum\limits_{i = 1}^{k - 1} {A_i^\beta {B_i}} \sum\limits_{i = 1}^{k - 1} {A_i^{2\alpha  - \beta  - 1}{C_i}}  - \sum\limits_{i = 1}^{k - 1} {A_i^\beta } \sum\limits_{i = 1}^{k - 1} {A_i^{2\alpha  - \beta  - 1}{B_i}{C_i}} \\
\quad  + A_k^{2\alpha  - 1}{B_k}{C_k} - A_k^{2\alpha  - 1}{B_k}{C_k}\\
\quad  + A_k^{2\alpha  - \beta  - 1}{C_k}\sum\limits_{i = 1}^{k - 1} {A_i^\beta {B_i}}  + A_k^\beta {B_k}\sum\limits_{i = 1}^{k - 1} {A_i^{2\alpha  - \beta  - 1}{C_i}} \\
\quad  - A_k^{2\alpha  - \beta  - 1}{B_k}{C_k}\sum\limits_{i = 1}^{k
- 1} {A_i^\beta }  - A_k^\beta \sum\limits_{i = 1}^{k - 1}
{A_i^{2\alpha  - \beta  - 1}{B_i}{C_i}} .
\end{array}\]
Based on the assumption for $k-1$, we have
\begin{align}
&\sum\limits_{i = 1}^k {A_i^\beta {B_i}} \sum\limits_{i = 1}^k
{A_i^{2\alpha  - \beta  - 1}{C_i}}  - \sum\limits_{i = 1}^k
{A_i^\beta } \sum\limits_{i = 1}^k {A_i^{2\alpha  - \beta  -
1}{B_i}{C_i}}
\\
& \le A_k^{2\alpha  - \beta  - 1}{C_k}\sum\limits_{i = 1}^{k - 1}
{A_i^\beta \left( {{B_i} - {B_k}} \right)}  - A_k^\beta
\sum\limits_{i = 1}^{k - 1} {A_i^{2\alpha  - \beta  - 1}{C_i}\left(
{{B_i} - {B_k}} \right)}\notag
\\
&= A_k^{2\alpha  - \beta  - 1}\sum\limits_{i = 1}^{k - 1}
{A_i^{2\alpha  - \beta  - 1}\left( {A_i^{2\beta  - 2\alpha  +
1}{C_k} - A_k^{2\beta  - 2\alpha  + 1}{C_i}} \right)\left( {{B_i} -
{B_k}} \right)} .\notag
\end{align} Noting ${\alpha ^2} \le 2\beta $
and $2\alpha  - 1 \le {\alpha ^2}$ implies $2\beta  - 2\alpha  + 1
\ge 0$.

If ${A_1} \ge {A_2} \ge  \cdot  \cdot  \cdot  \ge {A_k} > 0$, $0 <
{B_1} \le {B_2} \le  \cdot  \cdot  \cdot  \le {B_k}$, $0 < {C_1} \le
{C_2} \le  \cdot  \cdot  \cdot  \le {C_k}$, then for $i = 1, \cdot
\cdot \cdot ,k$,
\[\frac{{A_i^{2\beta  - 2\alpha  +
1}{C_k}}}{{A_k^{2\beta  - 2\alpha  + 1}{C_i}}} = \left(
{\frac{{{C_k}}}{{{C_i}}}} \right){\left( {\frac{{{A_i}}}{{{A_k}}}}
\right)^{2\beta  - 2\alpha  + 1}} \ge 1, \frac{{{B_i}}}{{{B_k}}} \le
1.\] and \begin{equation}A_i^{2\beta  - 2\alpha  + 1}{C_k} -
A_k^{2\beta - 2\alpha  + 1}{C_i} \ge 0, {\kern 1pt} {B_i} - {B_k}
\le 0.
\end{equation}
If ${A_i}, {B_i}$ and ${C_i},i = 1, \cdot \cdot
\cdot ,k$, are nonnegative, then (2.6) is also true. Hence from
(2.6),
\[\sum\limits_{i = 1}^k {A_i^\beta {B_i}} \sum\limits_{i = 1}^k {A_i^{2\alpha  - \beta  - 1}{C_i}}  - \sum\limits_{i = 1}^k {A_i^\beta } \sum\limits_{i = 1}^k {A_i^{2\alpha  - \beta  - 1}{B_i}{C_i}}  \le 0,\]
and (2.4) is proved.
\end{proof}

By Lemma 2.8, we immediately have the following result proved in
[1].
\begin{cor}
If ${A_1} \ge {A_2} \ge  \cdot \cdot  \cdot  \ge {A_k} \ge 0$, $0
\le {B_1} \le {B_2} \le  \cdot \cdot  \cdot  \le {B_k}$, $0 \le
{C_1} \le {C_2} \le  \cdot  \cdot \cdot  \le {C_k}$, $i = 1, \cdot
\cdot  \cdot ,k$, then we have\[\sum\limits_{i = 1}^n {A_i^2{B_i}}
\sum\limits_{i = 1}^n {{A_i}{C_i}}  \le \sum\limits_{i = 1}^n
{A_i^2} \sum\limits_{i = 1}^n {{A_i}{B_i}{C_i}}. \]
\end{cor}

Now let us describe some characteristic and noncharacteristic
domains with respect to vector fields and give some such domains.

\begin{definition}
Let $\phi (z,t)$ be the boundary function of a domain $\Omega $. We
call that a point $(z,t)$ on $\partial \Omega $ is a characteristic
point with respect to vector fields ${X_j},{Y_j}$ $(j = 1, \cdot
\cdot \cdot ,n)$, if it satisfies $\left| {{\nabla _L}\phi (z,t)}
\right| = 0$, where ${\nabla _L} = ({X_1}, \cdots .{X_n},{Y_1},
\cdots ,{Y_n})$. A domain with characteristic points is called a
characteristic domain. If the boundary $\partial \Omega $ does not
have any characteristic point, then $\Omega $ is said a
noncharacteristic domain.
\end{definition}
\begin{prop}
The sets ${\Omega _m} = \left\{ {(z,t) \in {C^{2n}} \times R\left|
{{{\left( {\left| z \right| - a} \right)}^2} + {{\left( {t - b}
\right)}^2} < {m^2}} \right.} \right\},$ $m = 1,2, \cdot  \cdot
\cdot ,$ are noncharacteristic domains with respect to
${X_j},{Y_j}(j = 1, \cdot  \cdot  \cdot ,n)$, where $a > 0$, $b$ is
any real number.
\end{prop}
\begin{proof}
Fix $m$ and denote $\psi (z,t) = {\left( {\left| z \right| - a}
\right)^2} + {\left( {t - b} \right)^2} - {m^2}$,
then\begin{equation}\begin{aligned} {X_j}\psi (z,t) &=
\frac{\partial }{{\partial {x_j}}}\left( {{{\left( {\left| z \right|
- a} \right)}^2} + {{\left( {t - b} \right)}^2} - {m^2}} \right) +
2\sigma {y_j}{\left| z \right|^{2\sigma  - 2}}\frac{\partial
}{{\partial t}}\left( {{{\left( {\left| z \right| - a} \right)}^2} +
{{\left( {t - b} \right)}^2} - {m^2}} \right)
\\
&= \frac{{2\left( {\left| z \right| - a} \right){x_j}}}{{\left| z
\right|}} + 4\sigma {y_j}{\left| z \right|^{2\sigma  - 2}}\left( {t
- b} \right),
\end{aligned}\notag\end{equation}
\begin{equation}\begin{aligned} {Y_j}\psi (z,t) &= \frac{\partial }{{\partial {y_j}}}\left( {{{\left( {\left| z \right| - a} \right)}^2} + {{\left( {t - b} \right)}^2} - {m^2}} \right) - 2\sigma {x_j}{\left| z \right|^{2\sigma  - 2}}\frac{\partial }{{\partial t}}\left( {{{\left( {\left| z \right| - a} \right)}^2} + {{\left( {t - b} \right)}^2} - {m^2}} \right)
\\
&= \frac{{2\left( {\left| z \right| - a} \right){y_j}}}{{\left| z
\right|}} - 4\sigma {x_j}{\left| z \right|^{2\sigma  - 2}}\left( {t
- b} \right)
\end{aligned}\notag\end{equation}
and
\begin{equation}\begin{aligned} {\left| {{\nabla _L}\psi (z,t)} \right|^2}
&= \sum\limits_{j = 1}^n {\left( {{{\left| {{X_j}\psi (z,t} \right|}^2} + {{\left| {{Y_j}\psi (z,t)} \right|}^2}} \right)}
\\
&= \sum\limits_{j = 1}^n {\left( {\frac{{4{{\left( {\left| z \right|
- a} \right)}^2}\left( {x_j^2 + y_j^2} \right)}}{{{{\left| z
\right|}^2}}} + 16{\sigma ^2}{{\left| z \right|}^{4\sigma  -
4}}{{\left( {t - b} \right)}^2}\left( {x_j^2 + y_j^2} \right)}
\right)} \\
& = 4{\left( {\left| z \right| - a} \right)^2} + 16{\sigma
^2}{\left| z \right|^{4\sigma  - 2}}{\left( {t - b} \right)^2}.
\end{aligned}\notag\end{equation}
If $\left| {{\nabla _L}\psi (z,t)} \right| = 0$, then $\left| z
\right| = a,t = b$. But points satisfying these conditions do not be
on the boundary $\partial {\Omega _m},m = 1,2, \cdot  \cdot  \cdot
$, so $\Omega _m,m = 1,2, \cdot  \cdot  \cdot $, are
noncharacteristic.
\end{proof}

If we take $a = 2,b = 0,$ then (see [2] for the case of Heisenberg
groups)
\begin{cor}
The sets ${\Omega _m} = \left\{ {(z,t) \in {C^{2n}} \times R\left|
{{{\left( {\left| z \right| - 2} \right)}^2} + {t^2} < {m^2}}
\right.} \right\}$, $m = 1,2, \cdots ,$ are noncharacteristic
domains with respect to vector fields ${X_j},{Y_j}(j = 1, \cdot
\cdot \cdot ,n)$. \end{cor}
\begin{prop}
The set $\Omega  = \left\{ {(z,t) \in {C^{2n}} \times R\left|
{\left| z \right|^{4\sigma }} + {t^2} < 1\right.} \right\}$ is a
characteristic domain with respect to vector fields ${X_j},{Y_j}(j =
1, \cdot \cdot \cdot ,n)$.
\end{prop}
\begin{proof}
Let $\varphi (z,t) = {\left| z \right|^{4\sigma }} + {t^2} - 1$,
then
\begin{equation}\begin{aligned}
{X_j}\varphi (z,t) &= \frac{\partial }{{\partial {x_j}}}\left(
{{{\left| z \right|}^{4\sigma }} + {t^2} - 1} \right) + 2\sigma
{y_j}{\left| z \right|^{2\sigma  - 2}}\frac{\partial }{{\partial
t}}\left( {{{\left| z \right|}^{4\sigma }} + {t^2} - 1} \right) \\
&=4\sigma {\left| z \right|^{4\sigma  - 2}}{x_j} + 4\sigma
{y_j}{\left| z \right|^{2\sigma  - 2}}t;
\end{aligned}\notag\end{equation}
\begin{equation}\begin{aligned}
{Y_j}\varphi (z,t) &= \frac{\partial }{{\partial {y_j}}}\left(
{{{\left| z \right|}^{4\sigma }} + {t^2} - 1} \right) - 2\sigma
{x_j}{\left| z \right|^{2\sigma  - 2}}\frac{\partial }{{\partial
t}}\left( {{{\left| z \right|}^{4\sigma }} + {t^2} - 1} \right) \\
&=4\sigma {\left| z \right|^{4\sigma  - 2}}{y_j} - 4\sigma
{x_j}{\left| z \right|^{2\sigma  - 2}}t.
\end{aligned}\notag\end{equation}
Hence
\begin{equation}\begin{aligned}
{\left| {{\nabla _L}\varphi (z,t)} \right|^2} &= \sum\limits_{j =
1}^n {\left( {{{\left| {{X_j}\varphi (z,t} \right|}^2} + {{\left|
{{Y_j}\varphi (z,t)} \right|}^2}} \right)}
\\
&= \sum\limits_{j = 1}^n {\left( {16{\sigma ^2}{{\left| z
\right|}^{8\sigma  - 4}}\left( {x_j^2 + y_j^2} \right) + 16{\sigma
^2}{{\left| z \right|}^{4\sigma  - 4}}{t^2}\left( {x_j^2 + y_j^2}
\right)} \right)} \\
&= 16{\sigma ^2}{\left| z \right|^{8\sigma  - 2}} + 16{\sigma
^2}{\left| z \right|^{4\sigma  - 2}}{t^2}.
\end{aligned}
\notag
\end{equation}
If $\left| {{\nabla _L}\phi (z,t)} \right| = 0$, then $\left| z
\right| = 0$. We see that two points satisfying $z = 0,t =  \pm 1$
are on the boundary $\partial \Omega $ and they are characteristic
points.
\end{proof}
\begin{cor}
The sets ${\Omega _r} = \left\{ {(z,t) \in {C^{2n}} \times R\left|
{{{\left| z \right|}^{4\sigma }} + {t^2} < {r^{4\sigma }}} \right.}
\right\} \left( {r > 0} \right)$ are characteristic domains with
characteristic points $\left( {0, \pm {r^{2\sigma }}} \right)$.
\end{cor}
\section{The proofs of Theorem 1.1 and Corollaries 1.2-1.4}
\noindent\textit{Proof of Theorem 1.1.}
We apply (2.3) with $A = L =
- {\Delta _L},$ ${B_1} = {x_1}, \cdot \cdot  \cdot ,{B_n} =
{x_n},{B_{n + 1}} = {y_1}, \cdot  \cdot  \cdot ,{B_{2n}} = {y_n},$
${T_1} = {X_1}, \cdot  \cdot  \cdot ,{T_n} = {X_n},{T_{n + 1}} =
{Y_1}, \cdot  \cdot  \cdot ,{T_{2n}} = {Y_n}$, $f(x) = {(\lambda  -
x)^\alpha }$, $g(x) = {(\lambda  - x)^\beta }$, and obtain
\begin{align}&{\left( {\sum\limits_{i = 1}^k
{\sum\limits_{p = 1}^n {{{({\lambda _{k + 1}} - {\lambda
_i})}^\alpha }\left( {{{\left\langle {\left[ {{X_p},{x_p}}
\right]{u_i},{u_i}} \right\rangle }_{{L^2}}} + {{\left\langle
{\left[ {{Y_p},{y_p}} \right]{u_i},{u_i}} \right\rangle }_{{L^2}}}}
\right)} } } \right)^2}
\\
 \le &4\left( {\sum\limits_{i = 1}^k
{\sum\limits_{p = 1}^n {{{({\lambda _{k + 1}} - {\lambda _i})}^\beta
}\left( {{{\left\langle {\left[ {L,{x_p}} \right]{u_i},{x_p}{u_i}}
\right\rangle }_{{L^2}}} + {{\left\langle {\left[ {L,{y_p}}
\right]{u_i},{y_p}{u_i}} \right\rangle }_{{L^2}}}} \right)} } }
\right) \notag
\\
&\times \left( {\sum\limits_{i = 1}^k {\sum\limits_{p = 1}^n
{{{({\lambda _{k + 1}} - {\lambda _i})}^{2\alpha  - \beta  -
1}}\left( {\left\| {{X_p}{u_i}} \right\|_{{L^2}}^2 + \left\|
{{Y_p}{u_i}} \right\|_{{L^2}}^2} \right)} } } \right).\notag
\end{align}
Since
\[\left[ {{X_p},{x_p}} \right]{u_i} = \left[ {{Y_p},{y_p}} \right]{u_i} = {u_i},\]
and
\[{\left\langle {\left[ {L,{x_p}} \right]{u_i},{x_p}{u_i}} \right\rangle _{{L^2}}} = 2\int_\Omega  {u_i^2}  - {\left\langle {\left[ {L,{x_p}} \right]{u_i},{x_p}{u_i}} \right\rangle _{{L^2}}}\]
from (2.2), it follows
\begin{equation}{\left\langle {\left[
{L,{x_p}} \right]{u_i},{x_p}{u_i}} \right\rangle _{{L^2}}} =
\int_\Omega {u_i^2}  = 1.
\end{equation}
In a similar way, we obtain
\begin{equation}
{\left\langle {\left[ {L,{y_p}} \right]{u_i},{y_p}{u_i}} \right\rangle _{{L^2}}} = \int_\Omega  {u_i^2}  =
1.
\end{equation}
On the other hand, it yields
\begin{equation}
\sum\limits_{p = 1}^n {\left\| {{X_p}{u_i}} \right\|_{{L^2}}^2}  + \sum\limits_{p = 1}^n {\left\| {{Y_p}{u_i}} \right\|_{{L^2}}^2}  = \int_\Omega  {{\nabla _L}{u_i}{\nabla _L}{u_i}}  = \int_\Omega  {L{u_i}{u_i}}  = \int_\Omega  {{\lambda _i}{u_i}{u_i}}  = {\lambda _i}.
\end{equation}
Instituting (3.2), (3.3)and (3.4) into (3.1), it deduces (1.4).
\hfill $\Box$

\noindent\textit{Proof of Corollary 1.2.} To obtain (1.5), we only
need to take $\alpha  = \beta  = 2$ in (1.4).
\hfill $\Box$

\noindent\textit{Proof of Corollary 1.3.} When $\alpha  = \beta $,
we have from Theorem 1.1 that
\[\sum\limits_{i = 1}^k {({\lambda _{k + 1}} - {\lambda _i}} {)^\alpha } \le \frac{2}{n}\sum\limits_{i = 1}^k {({\lambda _{k + 1}} - {\lambda _i}} {)^{\alpha  - 1}}{\lambda _i}.\]
Using Lemmas 2.6 and 2.7, it implies
\[\sum\limits_{i = 1}^k {({\lambda _{k + 1}} - {\lambda _i}} {)^\alpha } \ge \frac{1}{{{k^{\alpha  - 1}}}}{\left( {\sum\limits_{i = 1}^k {({\lambda _{k + 1}} - {\lambda _i}} )} \right)^\alpha } \ge {\left( {\sum\limits_{i = 1}^k {({\lambda _{k + 1}} - {\lambda _i}} )} \right)^{\alpha  - 1}}\left( {{\lambda _{k + 1}} - {\lambda _k}} \right)\]
and\[\frac{2}{n}\sum\limits_{i = 1}^k {({\lambda _{k + 1}} -
{\lambda _i}} {)^{\alpha  - 1}}{\lambda _i} \le \frac{2}{{nk}}\left(
{\sum\limits_{i = 1}^k {({\lambda _{k + 1}} - {\lambda _i}}
{)^{\alpha  - 1}}} \right)\left( {\sum\limits_{i = 1}^k {{\lambda
_i}} } \right),\] hence
\[{\left( {\sum\limits_{i = 1}^k {({\lambda _{k + 1}} - {\lambda _i}} )} \right)^{\alpha  - 1}}\left( {{\lambda _{k + 1}} - {\lambda _k}} \right) \le \frac{2}{{nk}}\left( {\sum\limits_{i = 1}^k {{{({\lambda _{k + 1}} - {\lambda _i})}^{\alpha  - 1}}} } \right)\left( {\sum\limits_{i = 1}^k {{\lambda _i}} } \right).\]
Since\[{\left( {\sum\limits_{i = 1}^k {({\lambda _{k + 1}} -
{\lambda _i}} )} \right)^{\alpha  - 1}} \ge \sum\limits_{i = 1}^k
{{{({\lambda _{k + 1}} - {\lambda _i})}^{\alpha  - 1}}}, \] it shows
(1.6).
\hfill $\Box$

\noindent\textit{Proof of Corollary 1.4.} When $1 \le \alpha  =
\beta  \le 2$, we have from Theorem 1.1 that
\[\sum\limits_{i = 1}^k {({\lambda _{k + 1}} - {\lambda _i}} {)^\alpha } \le \frac{2}{n}\sum\limits_{i = 1}^k {({\lambda _{k + 1}} - {\lambda _i}} {)^{\alpha  - 1}}{\lambda _i},\]
then
\begin{equation}\begin{aligned}
&{\lambda _{k +
1}}\sum\limits_{i = 1}^k {{{({\lambda _{k + 1}} - {\lambda
_i})}^{\alpha  - 1}}}  - \sum\limits_{i = 1}^k {{{({\lambda _{k +
1}} - {\lambda
_i})}^{\alpha  - 1}}{\lambda _i}}  \\
&= \sum\limits_{i = 1}^k
{{{({\lambda _{k + 1}} - {\lambda _i})}^{\alpha  - 1}}({\lambda _{k
+ 1}} - {\lambda _i})} \\
&\le \frac{2}{n}\sum\limits_{i = 1}^k {({\lambda _{k + 1}} -
{\lambda _i}} {)^{\alpha  - 1}}{\lambda _i},
\end{aligned}\notag\end{equation}
or
\begin{equation}\begin{aligned}
&{\lambda _{k + 1}}\sum\limits_{i = 1}^k {{{({\lambda _{k + 1}} -
{\lambda _i})}^{\alpha  - 1}}}  \\
&\le \left( {1 + \frac{2}{n}} \right)\sum\limits_{i = 1}^k
{({\lambda _{k + 1}} - {\lambda _i}} {)^{\alpha  - 1}}{\lambda _i}
\\
&\le \left( {1 + \frac{2}{n}} \right)\frac{1}{k}\left(
{\sum\limits_{i = 1}^k {({\lambda _{k + 1}} - {\lambda _i}}
{)^{\alpha  - 1}}} \right)\left( {\sum\limits_{i = 1}^k {{\lambda
_i}} } \right),
\end{aligned}\notag\end{equation}
where Lemma 2.7 is used. Therefore
\[\left( {{\lambda _{k + 1}} - \left( {1 + \frac{2}{n}} \right)\frac{1}{k}\left( {\sum\limits_{i = 1}^k {{\lambda _i}} } \right)} \right)\left( {\sum\limits_{i = 1}^k {{{({\lambda _{k + 1}} - {\lambda _i})}^{\alpha  - 1}}} } \right) \le 0.\]
Since $\left( {\sum\limits_{i = 1}^k {{{({\lambda _{k + 1}} -
{\lambda _i})}^{\alpha  - 1}}} } \right) \ge 0$, it follows
\[{\lambda _{k + 1}} - \left( {1 + \frac{2}{n}} \right)\frac{1}{k}\left( {\sum\limits_{i = 1}^k {{\lambda _i}} } \right) \le 0\]
and (1.7) is proved.
\hfill $\Box$

\section{Proofs of Theorem 1.5 and Corollaries 1.6-1.9}

\noindent\textit{Proof of Theorem 1.5.}
Applying (2.3) with $A =
{L^2} = {\left( { - {\Delta _L}} \right)^2},$ ${B_1} = {x_1}, \cdot
\cdot \cdot ,{B_n} = {x_n},{B_{n + 1}} = {y_1}, \cdot  \cdot  \cdot
,{B_{2n}} = {y_n},$ ${T_1} = {X_1}, \cdot  \cdot  \cdot ,{T_n} =
{X_n},{T_{n + 1}} = {Y_1}, \cdot \cdot  \cdot ,{T_{2n}} = {Y_n}$,
$f(x) = {(\lambda  - x)^\alpha }$, $g(x) = {(\lambda  - x)^\beta }$,
it follows
\begin{align}
&{\left( {\sum\limits_{i = 1}^k {\sum\limits_{p = 1}^n {{{({\lambda
_{k + 1}} - {\lambda _i})}^\alpha }\left( {{{\left\langle {\left[
{{X_p},{x_p}} \right]{u_i},{u_i}} \right\rangle }_{{L^2}}} +
{{\left\langle {\left[ {{Y_p},{y_p}} \right]{u_i},{u_i}}
\right\rangle }_{{L^2}}}} \right)} } } \right)^2}
\\
\le &4\left( {\sum\limits_{i = 1}^k {\sum\limits_{p = 1}^n
{{{({\lambda _{k + 1}} - {\lambda _i})}^\beta }\left(
{{{\left\langle {\left[ {{L^2},{x_p}} \right]{u_i},{x_p}{u_i}}
\right\rangle }_{{L^2}}} + {{\left\langle {\left[ {{L^2},{y_p}}
\right]{u_i},{y_p}{u_i}} \right\rangle }_{{L^2}}}} \right)} } }
\right) \notag
\\
&\times \left( {\sum\limits_{i = 1}^k {\sum\limits_{p = 1}^n
{{{({\lambda _{k + 1}} - {\lambda _i})}^{2\alpha  - \beta  -
1}}\left( {\left\| {{X_p}{u_i}} \right\|_{{L^2}}^2 + \left\|
{{Y_p}{u_i}} \right\|_{{L^2}}^2} \right)} } } \right).\notag
\end{align}
Since
\begin{equation}\begin{aligned}
&\sum\limits_{p = 1}^n {\left\| {{X_p}{u_i}} \right\|_{{L^2}}^2}  +
\sum\limits_{p = 1}^n {\left\| {{Y_p}{u_i}} \right\|_{{L^2}}^2}\\
&= \int_\Omega  {{\nabla _L}{u_i}{\nabla _L}{u_i}}= \int_\Omega
{L{u_i} \cdot {u_i}}\\
&\le {\left( {\int_\Omega  {u_i^2} } \right)^{\frac{1}{2}}}{\left(
{\int_\Omega  {{{\left( {L{u_i}} \right)}^2}} }
\right)^{\frac{1}{2}}} = \lambda _i^{\frac{1}{2}},
\end{aligned}\notag\end{equation}
it implies

\begin{align}
&\left( {\sum\limits_{i = 1}^k {\sum\limits_{p = 1}^n {{{({\lambda
_{k + 1}} - {\lambda _i})}^{2\alpha  - \beta  - 1}}\left( {\left\|
{{X_p}{u_i}} \right\|_{{L^2}}^2 + \left\| {{Y_p}{u_i}}
\right\|_{{L^2}}^2} \right)} } } \right)
\\
&= \left( {\sum\limits_{i = 1}^k {{{({\lambda _{k + 1}} - {\lambda
_i})}^{2\alpha  - \beta  - 1}}\lambda _i^{\frac{1}{2}}} }
\right).\notag
\end{align}

Recalling (3.2) and (3.3), we have

\begin{align}
&{\left( {\sum\limits_{i = 1}^k {\sum\limits_{p = 1}^n {{{({\lambda
_{k + 1}} - {\lambda _i})}^\alpha } \left( {{{\left\langle {\left[
{{X_p},{x_p}} \right]{u_i},{u_i}} \right\rangle }_{{L^2}}} +
{{\left\langle {\left[ {{Y_p},{y_p}} \right]{u_i},{u_i}}
\right\rangle }_{{L^2}}}} \right)} } } \right)^2}
\\
&= 4{n^2}{\left( {\sum\limits_{i = 1}^k {{{({\lambda _{k + 1}} -
{\lambda _i})}^\alpha }} } \right)^2}.\notag
\end{align}
On the other hand, it obtains by (2.2) that
\begin{equation}
\begin{aligned}
\left[ {{L^2},{x_p}} \right]{u_i} &= {L^2}\left( {{x_p}{u_i}}
\right) - {x_p}{L^2}{u_i} \\
&=  - 2{X_p}L{u_i} - 2L\left( {{X_p}{u_i}} \right),
\end{aligned}\notag
\end{equation}
and
\[ \left[ {{L^2},{y_p}} \right]{u_i}=  - 2{Y_p}L{u_i} - 2L\left( {{Y_p}{u_i}} \right).\]
Hence, we have
\begin{equation}\begin{aligned}
{\left\langle {\left[ {{L^2},{x_p}} \right]{u_i},{x_p}{u_i}}
\right\rangle _{{L^2}}} &= 2\int_\Omega  {L{u_i} \cdot {X_p}\left(
{{x_p}{u_i}} \right)}  - 2\int_\Omega  {{x_p}{X_p}{u_i} \cdot
L{u_i}}  - 4\int_\Omega  {X_p^2{u_i} \cdot
{u_i}}\\
&= 2\int_\Omega  {L{u_i} \cdot {u_i}}  - 4\int_\Omega {X_p^2{u_i}
\cdot {u_i}}
\end{aligned}\notag\end{equation}
and
\begin{equation}\begin{aligned}
{\left\langle {\left[ {{L^2},{y_p}} \right]{u_i},{y_p}{u_i}}
\right\rangle _{{L^2}}} &= 2\int_\Omega  {L{u_i} \cdot {Y_p}\left(
{{y_p}{u_i}} \right)}  - 2\int_\Omega  {{y_p}{Y_p}{u_i} \cdot
L{u_i}}  - 4\int_\Omega  {Y_p^2{u_i} \cdot
{u_i}}\\
&= 2\int_\Omega  {L{u_i} \cdot {u_i}}  - 4\int_\Omega {Y_p^2{u_i}
\cdot {u_i}}.
\end{aligned}\notag\end{equation}
Noting
\[ - \sum\limits_{p = 1}^n {\int_\Omega  {X_p^2{u_i} \cdot {u_i}} }  - \sum\limits_{p = 1}^n {\int_\Omega  {Y_p^2{u_i} \cdot {u_i}} }  = \sum\limits_{p = 1}^n {\left\| {{X_p}{u_i}} \right\|_{{L^2}}^2}  + \sum\limits_{p = 1}^n {\left\| {{Y_p}{u_i}} \right\|_{{L^2}}^2}  = \int_\Omega  {L{u_i} \cdot {u_i}} ,\]
so
\begin{align} &\sum\limits_{i = 1}^k {\sum\limits_{p = 1}^n
{{{({\lambda _{k + 1}} - {\lambda _i})}^\beta }\left(
{{{\left\langle {\left[ {{L^2},{x_p}} \right]{u_i},{x_p}{u_i}}
\right\rangle }_{{L^2}}} + {{\left\langle {\left[ {{L^2},{y_p}}
\right]{u_i},{y_p}{u_i}} \right\rangle }_{{L^2}}}} \right)} }
\\
=& \sum\limits_{i = 1}^k {\sum\limits_{p = 1}^n {{{({\lambda _{k +
1}} - {\lambda _i})}^\beta }\left( {2\int_\Omega  {L{u_i} \cdot
{u_i}}  - 4\int_\Omega  {X_p^2{u_i} \cdot {u_i}} } \right)} }\notag
\\
&+ \sum\limits_{i = 1}^k {\sum\limits_{p = 1}^n {{{({\lambda _{k +
1}} - {\lambda _i})}^\beta }\left( {2\int_\Omega  {L{u_i} \cdot
{u_i}}  - 4\int_\Omega  {Y_p^2{u_i} \cdot {u_i}} } \right)} }\notag
\\
=& 4\left( {n + 1} \right)\sum\limits_{i = 1}^k {{{({\lambda _{k +
1}} - {\lambda _i})}^\beta }\int_\Omega  {L{u_i} \cdot {u_i}}
}\notag
\\
\le& 4\left( {n + 1} \right)\sum\limits_{i = 1}^k {{{({\lambda _{k +
1}} - {\lambda _i})}^\beta }\lambda _i^{\frac{1}{2}}} .\notag
\end{align}
Taking (4.2), (4.3) and (4.4) into (4.1), we obtain (1.8). \hfill
$\Box$

\noindent\textit{Proof of Corollary 1.6.} To obtain (1.9), take
$\alpha = \beta  = 2$ in (1.8).
\hfill $\Box$

\noindent\textit{Proof of Corollary 1.7.}
From Theorem 1.5, we have
\[{\left( {\sum\limits_{i = 1}^k {{{({\lambda _{k + 1}} - {\lambda _i})}^\alpha }} } \right)^2} \le \frac{{4\left( {n + 1} \right)}}{{{n^2}}}\left( {\sum\limits_{i = 1}^k {{{({\lambda _{k + 1}} - {\lambda _i})}^\beta }\lambda _i^{\frac{1}{2}}} } \right)\left( {\sum\limits_{i = 1}^k {{{({\lambda _{k + 1}} - {\lambda _i})}^{2\alpha  - \beta  - 1}}\lambda _i^{\frac{1}{2}}} } \right).\]
Applying Lemma 2.8 with ${A_i} = {\lambda _{k + 1}} - {\lambda _i}$
and ${B_i} = {C_i} = \lambda _i^{\frac{1}{2}}$, it deduces (1.10).
\hfill $\Box$

\noindent\textit{Proof of Corollary 1.8.} To obtain (1.11), we only
need to take $\alpha  = \beta  = 2$ in Corollary 1.7.
\hfill $\Box$

\noindent\textit{Proof of Corollary 1.9.}
We have from (1.8) that
\[{\left( {\sum\limits_{i = 1}^k {{{({\lambda _{k + 1}} - {\lambda _i})}^\alpha }} } \right)^2} \le \frac{{4\left( {n + 1} \right)}}{{{n^2}}}\left( {\sum\limits_{i = 1}^k {{{({\lambda _{k + 1}} - {\lambda _i})}^\beta }\lambda _i^{\frac{1}{2}}} } \right) \times \left( {\sum\limits_{i = 1}^k {{{({\lambda _{k + 1}} - {\lambda _i})}^{2\alpha  - \beta  - 1}}\lambda _i^{\frac{1}{2}}} } \right).\]
Applying Lemma 2.7 to $\left( {\sum\limits_{i = 1}^k {{{({\lambda
_{k + 1}} - {\lambda _i})}^\beta }\lambda _i^{\frac{1}{2}}} }
\right)$ and $\left( {\sum\limits_{i = 1}^k {{{({\lambda _{k + 1}} -
{\lambda _i})}^{2\alpha  - \beta  - 1}}\lambda _i^{\frac{1}{2}}} }
\right)$, it follows
\begin{equation}\begin{aligned}
&{\left( {\sum\limits_{i = 1}^k {{{({\lambda _{k + 1}} - {\lambda
_i})}^\alpha }} } \right)^2}\\
&\le \frac{{4\left( {n + 1} \right)}}{{{n^2}{k^2}}}\left(
{\sum\limits_{i = 1}^k {{{({\lambda _{k + 1}} - {\lambda _i})}^\beta
}} } \right)\left( {\sum\limits_{i = 1}^k {{{({\lambda _{k + 1}} -
{\lambda _i})}^{2\alpha  - \beta  - 1}}} } \right){\left(
{\sum\limits_{i = 1}^k {\lambda _i^{\frac{1}{2}}} } \right)^2}\\
&= \frac{{4\left( {n + 1} \right)}}{{{n^2}{k^2}}}\left(
{\sum\limits_{i = 1}^k {{{({\lambda _{k + 1}} - {\lambda
_i})}^\alpha }} } \right) \times \left( {\sum\limits_{i = 1}^k
{{{({\lambda _{k + 1}} - {\lambda _i})}^{\alpha  - 1}}} }
\right){\left( {\sum\limits_{i = 1}^k {\lambda _i^{\frac{1}{2}}} }
\right)^2},
\end{aligned}\notag\end{equation}
where we have used $1 \le \alpha  = \beta  \le 2$. It implies
\[\sum\limits_{i = 1}^k {{{({\lambda _{k + 1}} - {\lambda _i})}^\alpha }}  \le \frac{{4\left( {n + 1} \right)}}{{{n^2}{k^2}}}\left( {\sum\limits_{i = 1}^k {{{({\lambda _{k + 1}} - {\lambda _i})}^{\alpha  - 1}}} } \right){\left( {\sum\limits_{i = 1}^k {\lambda _i^{\frac{1}{2}}} } \right)^2},\]
then
\[\sum\limits_{i = 1}^k {{{({\lambda _{k + 1}} - {\lambda _i})}^{\alpha  - 1}}\left( {\left( {{\lambda _{k + 1}} - {\lambda _k}} \right) - \frac{{4\left( {n + 1} \right)}}{{{n^2}{k^2}}}{{\left( {\sum\limits_{i = 1}^k {\lambda _i^{\frac{1}{2}}} } \right)}^2}} \right)}  \le 0,\]
since ${\lambda _i} \le {\lambda _k}$ for all $i \le k$. Hence
\[\left( {{\lambda _{k + 1}} - {\lambda _k}} \right) - \frac{{4\left( {n + 1} \right)}}{{{n^2}{k^2}}}{\left( {\sum\limits_{i = 1}^k {\lambda _i^{\frac{1}{2}}} } \right)^2} \le 0,\]
and (1.12) is proved.
\hfill $\Box$

\end{document}